\documentclass[10pt]{article}

\usepackage[all]{xy}
\usepackage{amsmath} 
\usepackage{amssymb}
\usepackage{theorem}
\usepackage{pst-all}

\newtheorem{theorem}{Theorem}[section]
\newtheorem{lemma}[theorem]{Lemma}
\newtheorem{proposition}[theorem]{Proposition}
\newtheorem{corollary}[theorem]{Corollary}

\theoremstyle{definition}
\newtheorem{definition}[theorem]{Definition}
\newtheorem{example}[theorem]{Example}

\theoremstyle{remark}
\newtheorem{remark}[theorem]{Remark}

\numberwithin{equation}{section}

\newenvironment{proof}
  {\begin{trivlist}
  \item[\textit{\noindent\textsc{ Proof.}}]}
  {{$\square$}\end{trivlist}}

\newcommand{\Adj}{\ensuremath{\mbox{\rm Adj}}}

\newcommand{\Gal}{\ensuremath{\mbox{\rm Gal}}}

\newcommand{\Gr}{\ensuremath{\mbox{\rm Gr}}}

\begin{document}

\title{Lie's Reduction Method and Differential Galois Theory in the Complex Analytic Context
 }
\author{David Bl\'azquez-Sanz \& Juan Jos\'e Morales-Ruiz\footnote{This research has been partially
 supported by grant MCyT-FEDER MTM2006-00478of Spanish goverment, and the Sergio Arboleda University 
Research Agency CIVILIZAR.}}


\maketitle

\begin{abstract}
  This paper is dedicated to the differential Galois theory in the
complex analytic context for Lie-Vessiot systems. Those are
the natural generalization of linear systems, and the more 
general class of differential equations adimitting superposition
laws, as recently stated in \cite{BM2008}. A Lie-Vessiot system
is automatically translated into a equation in a Lie group
that we call automorphic system. Reciprocally an automorphic system
induces a hierarchy of Lie-Vessiot systems. In this work we 
study the global analytic aspects of a classical method of reduction
of differential equations, due to S. Lie. We propose an differential
Galois theory for automorphic systems, and explore the relationship
between integrability in terms of Galois theory and the Lie's reduction
method. Finally we explore the algebra of Lie symmetries of a general
automorphic system. 
\end{abstract}



\section*{Summary of results}
  
  Throughout the first section we review the concept of Lie-Vessiot system,
and we state the global version of Lie's superposition theorem; it is a
recallment of results contained in previous works \cite{BM2008} and \cite{B2008}. 
Second section is devoted to the notion of automorphic system. We study its 
geometry and the relationship with general Lie-Vessiot systems: it is introduced 
the concept of Lie-Vessiot hierarchy, that relates an automorphic system with 
a family of Lie-Vessiot systems induced in homogeneous spaces. In the third 
section we introduce Lie's reduction method (Theorem \ref{C1THEliereduction}) 
in the complex analytic context. We note that
Lie's reduction method is local in the time parameter, and then we
explore the obstructions to the existence of a reduction method global in time
(Proposition \ref{C1PROmeromorphicliereduction}). 
In the fourth section we propose a Galois theory for automorphic systems,
based on the Lie reduction method. In our theory, the Galois group is the
smallest group to which the equation can be reduced. We explore the relationship
with the classical Picard-Vessiot theory. We prove that our Galois group
iz Zariski dense in the classical one (Theorem \ref{TH2.18}). We study the reduction 
of the automorphic system to
the Galois group (Theorem \ref{ATkolchin}), which is an geometric analytic version of the
classical theorem of reduction due to Kolchin, and the integrability by quadratures in 
solvable groups (Theorem \ref{ATsolvable}). Finally, in the last section we study the infinitesimal symmetries of 
automorphic systems under the light of differential Galois theory. 
We give a characterization of the algebra of Lie symmetries generalizing
a result due to Athorne (Theorem \ref{C1THE1.5.9}). We also find that the Lie algebra
of meromorphic right invariant symmetries is contained in the centralizer algebra
of the Galois group (Theorem \ref{C1THE1.5.11}), which is a non symplectic version of the Morales-Ramis
lemma, which follows as a corollary.

\section{Lie-Vessiot systems}

  The class of ordinary differential equations admitting
\emph{fundamental systems of solutions} was introduced by S. Lie in
1885 \cite{Lie1885}, as certain class of auxiliary equations
appearing in his integration methods for ordinary differential
equations. An ordinary differential equation admitting a \emph{fundamental
system of solutions} is, by definition, a system of non-autonomous
differential equations,
\begin{equation}\label{ATeq1}
  \frac{dx_i}{dt} = F_i(t,x_1,\dots,x_n) \quad i=1,\ldots,n
\end{equation}
for which there exists a set of formulae,
\begin{equation}\label{ATeq2}
\varphi_i(x^{(1)},\ldots,x^{(r)},\lambda_1,\ldots,\lambda_n) \quad
i=1,\ldots,n,
\end{equation}
expressing the \emph{general solution} as function of $r$ particular
solutions of \eqref{ATeq1} and some arbitrary constants $\lambda_i$.
This means that for $r$ particular solutions $x^{(i)}(t)$ of the
equations satisfying certain non-degeneracy condition, the expression:
\begin{equation}\label{ATeq3}
  x_i(t) = \varphi_i(x^{(1)}(t),\ldots,x^{(r)}(t),\lambda_1,\ldots,\lambda_n)
\end{equation}
is the general solution of the equation \eqref{ATeq1}. In
\cite{Lie1893b} Lie also stated that the arbitrary constants
$\lambda_i$ parameterize the solution space, in the sense that for
different constants, we obtain different solutions: there are not
functional relations between the arbitrary constants $\lambda_i$.
The set of formulae $\varphi_i$ is usually referred to as a
\emph{superposition law for solutions} of \eqref{ATeq1}.

Lie's superposition theorem \eqref{ATeq1} states that a differential
equation \emph{locally} admits a superposition law if and only if it 
has a finite dimensional Lie-Vessiot-Guldberg algebra (Definition \ref{DefLVG}). It was believed 
that the integration of this Lie algebra to a Lie group
action lead to a global superposition law. Nowadays it is known that
there are some other geometrical obstructions to the existence 
of a superposition law going beyond the integration of the Lie-Vessiot-Guldberg 
algebra: in a recent work \cite{BM2008} a characterization of differential
equations admitting superpositon laws is given.

\subsection{Non-autonomous Analytic Vector Fields}

\begin{definition}
  A non-autonomous complex analytic vector field $\vec X$ in $M$, depending on
  the Riemann surface $S$, is an autonomous vector field in $S\times M$, compatible with $\partial$
  in the following sense:
  $$\vec X f(t) = \partial f(t)\quad\quad
  \xymatrix{\mathcal O_{S\times M} \ar[r]^-{\vec X} & \mathcal O_{S\times M} \\
  \mathcal O_S \ar[r]^-{\partial}\ar[u] & \mathcal O_S \ar[u]}$$
\end{definition}

  In each cartesian power $M^r$ of $M$ we consider the \emph{lifted} vector
field $\vec X^r$. This is just the direct sum copies of $\vec X$
acting in each component of the cartesian power $M^r$. We have the
local expression for $\vec X$,
$$\vec X = \partial + \sum_{i=1}^n F_i(t,x) \frac{\partial}{\partial
x_i},$$ and also the local expression for $\vec X^r$, which is a
non-autonomous vector field in $M^r$,
\begin{equation}\label{EQlocalX}
\vec X^r = \partial + \sum_{i=1}^n
F_i(t,x^{(1)})\frac{\partial}{\partial x_i^{(1)}} + \ldots +
\sum_{i=1}^n F_i(t,x^{(r)})\frac{\partial}{\partial x_i^{(r)}}.
\end{equation}

\subsection{Superposition Law}

\begin{definition}
  A superposition law for $\vec X$ is an analytic map
  $$\varphi\colon U \times M \to M,$$
  where $U$ is analytic open subset of $M^r$, verifying:
  \begin{itemize}
  \item[(a)] $U$ is union of integral curves of $\vec X^r$.
  \item[(b)] If $\bar x(t)$ is a solution of $\vec X^r$, defined for $t$
  in some open subset $S'\subset S$, then $x_{\lambda}(t) =\varphi(\bar
  x(t), \lambda)$, where $\lambda$ varies in $M$, is the general solution of $\vec X$ for $t$ varying in
  $S'$.
  \end{itemize}
\end{definition}

\begin{example}[Linear systems]
Let us consider a linear system of ordinary differential equations,
$$\frac{dx_i}{dt} = \sum_{j=1}^n a_{ij}(t)x_j, \quad i=1,\ldots,n$$
linear combinations of solutions of this system
are also a solutions. Thus, the solution of the system is a $n$
dimensional vector space, and we can express the global solution as
linear combinations of $n$ linearly independent solutions. The
superposition law is written down as follows,
$$\mathbb C^{n\times n}_{x_i^{(j)}} \times \mathbb C^n_{\lambda_j} \to \mathbb
C^n, \quad (x_i^{(j)},\lambda_j) \mapsto (y_i) \quad y_i =
\sum_{j=1}^n \lambda_jx_i^{(j)}.$$
\end{example}

\begin{example}[Riccati equations]
Let us consider the ordinary differential equation,
$$\frac{dx}{dt} = a(t) + b(t)x + c(t)x^2,$$
let us consider four different solutions
$x_1(t),x_2(t),x_3(t),x_4(t)$. A direct computation gives that the
anharmonic ratio is constant,
$$\frac{d}{dt} \frac{(x_1 - x_2)(x_3-x_4)}{(x_1-x_4)(x_3-x_2)} =
0.$$If $x_1, x_2, x_3$ represent three known solutions, we can
extract the unknown solution $x$ of the expression,
$$\lambda = \frac{(x_1 - x_2)(x_3-x)}{(x_1-x)(x_3-x_2)}$$
obtaining,
$$x = \frac{x_3(x_1-x_2)-\lambda
x_1(x_3-x_2)}{(x_1-x_2)-\lambda(x_3-x_2)}$$ which is the general
solution in function of the constant $\lambda$, and then a
superposition law for the Riccati equation.
\end{example}

\subsection{Lie's Superposition Theorem}

\begin{definition}\label{DefLVG}
  The Lie-Vessiot-Guldberg algebra of $\vec X$ is the Lie algebra of
  vector fields in $M$ spanned by the set vector fields $\{\vec
  X_{t_0}\}_{t_0\in S}$. The Lie-Vessiot-Guldberg algebra of $\vec
  X$ is denoted $\mathfrak g(\vec X)$.
\end{definition}

\begin{remark}
  The Lie-Vessiot-Guldberg algebra of $\vec X$ is finite dimensional
  if and only if there exist $\vec X_1,\ldots, \vec X_s$ autonomous
  vector fields in $M$, spanning a finite dimensional Lie algebra,
  and holomorphic functions $f_1(t),\ldots, f_s(t)$ in $S$ such that,
  $$\vec X = \partial + \sum_{i=1}^s f_i(t)\vec X_i.$$
\end{remark}

{\bf Notation. } From now on, let us consider a complex analytic Lie
group $G$, and a \emph{faithful} analytic action of $G$ on $M$,
$$G\times M \to M, \quad (\sigma, x) \to \sigma \cdot x.$$
$\mathcal R(G)$ is the Lie algebra of \emph{right-invariant} vector
fields in $G$. We denote by $\mathcal R(G,M)$ the Lie algebra of
fundamental vector fields of the action of $G$ on $M$ (see, for
instance, \cite{Nomizu}).

\medskip

A point $x$ of $M$ is called a \emph{principal point} if its istropy
subgroup $H_x$ consist in the identity element only. It is clear that
$x$ is a principal point if and only if its orbit $O_x$ is a \emph{principal
homogeneous $G$-space}. For each $r\in\mathbb N$, $G$ 
acts in the cartesian power $M^r$ component by component. There is a minimum $r$ 
such that there exist principal orbits in $M^r$. If $M$ is an homogeneous space then this
number $r$ is called the \emph{rank} of $M$. Most homogeneous spaces
are of finite rank. For instance, algebraic homogeneous spaces are (see
\cite{B2008, BM2008}). 

\begin{definition}\label{C1DEFpretransitive}
We say that the action of $G$ is on $M$ pretransitive if there
exists $r\geq 1$ and an analytic open subset $U\subset M^r$ such
that:
\begin{enumerate}
\item[(a)] $U$ is union of principal orbits.
\item[(b)] The space of orbits $U/G$ is a complex analytic manifold.
\end{enumerate}
\end{definition}

  In particular, finite rank $G$-homogeneous spaces are
pretransitive (see \cite{B2008, BM2008}). The notions of Lie-Vessiot
sytem (below) and pretransitive action characterize the
differential equations admitting superposition laws. 

\begin{definition}
  A non-autonomous analytic vector field $\vec X$ in $M$ is called a Lie-Vessiot system,
  relative to the action of $G$, if its Lie-Vessiot-Guldberg algebra is spanned by fundamental
  fields of the action of $G$ on $M$; if and only if
  $\mathfrak g(\vec X)\subset \mathcal R(G,M)$.
\end{definition}

\begin{theorem}[global Lie's superposition theorem 	\cite{B2008, BM2008}]\label{C1THEgloballie}
The non- autonomous vector field $\vec X$ in $M$ admits a
superposition law if and only if it is a Lie-Vessiot system related
to certain pretransitive Lie group action in $M$.
\end{theorem}

  An interesting point to remark is that the group underlying a Lie-Vessiot
system is recovered from the superposition law (see  \cite{BM2008}, Lemma 6.4)

\section{Automorphic Systems}

 The \emph{automorphic system} is the translation of a Lie-Vessiot system
to the Lie group $G$. This approach is due to Vessiot. From now on,
consider the a complex analytic connected Lie group $G$, and a
faithful pre-transitive action of $G$ on  $M$.

\medskip

  Let $\vec X$  be a Lie-Vessiot system in $M$ with coefficients in
the Riemann surface $S$. Then,
  $$\vec X = \partial + \sum_{i=1}^s f_i(t)\vec X_i,$$
where the vector fields $\vec X_i$ are fundamental vector fields in
$M$. Consider the natural map,
$$\mathcal R(G) \to \mathfrak X(M), \quad \vec A \mapsto \vec A^M,$$
applying right invariant vector fields to fundamental fields. Let us
call $\vec A_i$ to the element of $\mathcal R(G)$ such that $\vec
A_i^M = \vec X_i$.

\begin{definition}
  We call automorphic vector field associated to $\vec X$ to the
non-autonomous vector field in $G$,
$$\vec A = \partial + \sum_{i=1}^s f_i(t)\vec A_i.$$
 Reciprocally let $N$ be an homogeneous $G$-space; we call
Lie-Vessiot system induced in $N$ by $\vec A$ to the non-autonomous
vector field,
$$\vec A^N = \partial + \sum_{i=1}^s f_i(t) \vec A^N_i.$$
\end{definition}

  Let us note that $\vec X$ is the Lie-Vessiot system $\vec
  A^M$ induced in $M$ by $\vec A$.

\subsection{Superposition Law for the Automorphic System}

  The action of $G$ on itself is transitive. Right invariant vector fields in $G$ are
fundamental fields in $G$. Then, the automorphic system $\vec A$ is
a particular case of a Lie-Vessiot system. Hence, there is a
superposition law for $\vec A$.

  Consider $\sigma(t)$ a local solution of $\vec A$. At $t_0\in
  S$, the tangent vector $\sigma'(t_0)$ is $(\vec
  A_{t_0})_{\sigma(t_0)}$. Consider any $\tau\in G$; let us define a
  new curve $\gamma(t)$ in $G$ as the composition $\sigma(t)\cdot
  \tau$. The tangent vector to the curve $\gamma(t)$ at $t_0$ is
  $\gamma'(t_0) = R_{\tau}'(\sigma'(t_0)) = R_{\tau}'((\vec A_{t_0})_{\sigma(t_0)})
  = (\vec A_{t_0})_{\gamma(t_0)}$. Hence, $\gamma(t)$ is another
  solution of $\vec A$.

\begin{proposition}
  The composition law $G\times G \to G$ in $G$ is the superposition
  principle for $\vec A$.
\end{proposition}

\begin{proof}
  Consider a solution $\sigma(t)$. As stated above, for all $\tau$
in $G$, $\sigma(t)\cdot\tau$ is another solution of $\vec A$. Let us
see that this is the general solution. Consider $t_0$ in $S$. For
each $\tau$ in $G$, $\sigma(t)\cdot\tau$ is the solution curve of
initial conditions $t_0\mapsto \sigma(t_0)\cdot\tau$. The action of
$G$ on itself -by the right side- is free an transitive; and all
initial conditions are obtained in this way.
\end{proof}

\begin{corollary}
  Consider $\sigma(t)$ and $\tau(t)$ two solutions of $\vec A$.
  Then $\sigma(t)\cdot\tau(t)^{-1}$ is a constant point of $G$.
\end{corollary}

\subsection{Structure of Solution Space}

  The particularity of Lie-Vessiot systems is that certain finite
sets of solutions give us the general solution. For the automorphic
systems, this structure is even simpler. Any particular solution
gives us the general solution. There is no difference between the
notion of general and particular solution. 

\begin{proposition}
Let $\vec A$ be an automorphic system in $G$ depending on $S$.
Consider $S'\subset S$ such that there exist an analytic solution
$\sigma\colon S'\to G$. Then the space of solutions of $\vec A$
defined in $S'$ is a principal homogeneous space with an action of
$G$ by the right side.
\end{proposition}

\begin{proof}
Consider $Sol(\vec A)$ the space of solutions of $\vec A$ defined in
$S$. The superposition law gives us have an action of $G$ on
$Sol(\vec A)$ by the right side,
$$Sol(\vec A) \times G \to Sol(\vec A), \quad (\sigma(t),\tau) \to
R_{\tau}(\sigma(t)).$$ This action is free and transitive, by
uniqueness of solutions. The space of solutions $Sol(\vec A)$ is a
principal homogeneous space.
\end{proof}

\subsection{Hierarchy of Lie-Vessiot Systems} 

  Let us consider the following objects: two $G$-spaces
$M$ and $N$, an automorphic vector field $\vec A$ in $G$ depending
on the Riemann surface $S$, and the induced Lie-Vessiot systems
$\vec A^M$ and $\vec A^N$ in $M$ and $N$ respectively.

\quad

  Let $f$ be a surjective morphism of $G$-spaces,
  $$f\colon M \to N, \quad f(\sigma\cdot x) = \sigma \cdot f(x).$$
The map $f$ applies fundamental vector fields of the action of $G$
on $M$ to fundamental vector fields of the action of $G$ on $N$.
Thus, $f$ transforms Lie-Vessiot systems in $\vec M$ into
Lie-Vessiot systems in $N$. It is clear that it transforms $\vec
A^M$ into $\vec A^N$:
$$f_*(\vec A^M) = \vec A^N.$$

  A solution curve $x(t)$ of $\vec A^N$ gives us, by composition, a
solution curve $f(x(t))$ of $\vec A^M$. We have a surjective map:
$$Sol(\vec A^M) \to Sol(\vec A^N).$$

  Let us assume that $\vec A^M$ admits a superposition
law $\varphi$, which is true if the action of $G$ in $M$ is
pretransitive,
$$\varphi \colon U \times M \to M, \quad U \subset M^r.$$
We can also use this superposition law of $\vec A^M$ in $M$ for
expressing the general solution of $\vec A^N$. The composition
$\phi_1 = f\circ\varphi$ express the general solution of $\vec A^N$
as function of $r$ particular solutions of $\vec A^M$,
$$\phi_1\colon U \times M \to N,\quad U\subset M^r.$$
Consider $\bar x(t)$, and $x, y\in M$ such that $f(x)= f(y)$. Then
$\phi_1(\bar x(t),x)=\phi_1(\bar x(t), y)$. Then $\phi_1$ factorizes
and gives un a map,
\begin{equation}\label{ATeq23}
\phi\colon U \times N \to N, \quad  U \subset M^r
\end{equation}
that gives the general solution of $\vec A^N$ in function of $r$
particular solutions of $\vec A^M$. It is not a superposition law,
but a different object known as a \emph{representation formula} \cite{ShorineWinternitz}. A
general theory of representation formulae of solutions of
differential equations can be done under this point of view.

\medskip

  In particular, the associative property of the action
$a\colon G\times M\to M$ means that it is a morphism of $G$ spaces. 
It is clear that $a$ sends the Lie-Vessiot system $\vec
A$ in $G\times M$ to the Lie-Vessiot system $\vec A^M$.

\medskip

  If $M$ is a $G$-homogeneous space, then there is a surjective
morphisms of $G$-spaces,
$$f\colon G \to M, \quad \sigma \mapsto \sigma\cdot x_0.$$
In this particular case, the representation formula \eqref{ATeq23}
gives us that a particular solution of $\vec A$ gives us the general
solution of $\vec A^M$ thorough the action of $G$ on $M$. If
$\sigma(t)$ is a solution of $\vec A$, then $\sigma(t)\cdot x_0$ is
the general solution of $\vec X$ when $x_0$ moves in $M$. In
particular, if we take $\sigma(t)$ such that $\sigma(t_0) = Id$,
then $\sigma(t)\cdot x_0$ is the solution of the problem of initial
conditions $x(t_0)=x_0$.


\begin{example}[Linear systems]
Let us consider the system,
$$\frac{dx_i}{dt} = \sum_{j=1}^n a_{ij}(t)x_j.$$
As a vector field it is written
$$\vec X = \sum_{i,j} a_{ij}(t) \vec X_{ij}, \quad \vec X_{ij} =
x_j\frac{\partial}{\partial x_i}.$$ these vector fields span the lie
algebra $gl(n,\mathbb C)$ of the action of general linear group
$GL(n,\mathbb C)$ on $\mathbb C^n$. Let us take coordinates $u_{ij}$
in $GL(n,\mathbb C)$. Direct computation gives us the right
invariant vector fields,
$$\vec A_{ij} = \sum_{k=1}^n u_{jk}\frac{\partial}{\partial
u_{ik}}$$ then,
$$\vec A = \partial + \sum_{i,j} a_{ij}(t) \vec A_{ij}.$$
The automorphic system is written,
$$\frac{d u_{ij}}{dt} = \sum_k a_{ik}(t) u_{kj},$$
or in matrix form,
$$\frac{d}{dt} U = A(t)U, \quad\quad U = (u_{ij}),\, A(t) = (a_{ij}(t)).$$
The solutions of the automorphic system are the fundamental matrices
of solutions of the linear system. If $U(t)$ is one such of these
matrices, then for each $x_0\in\mathbb C^n$, $x(t) = U(t).x_0$ is a
solution of the linear system. Moreover if we take $U(t)$ such that
$U(t_0) = Id$ the previous formula gives us the global solution with
initial conditions $x(t_0) = x_0$.
\end{example}

\begin{example}[Riccati equations]
Let us consider the general Riccati equation,
$$\frac{dx}{dt} = a(t) + b(t)x + c(t)x^2.$$
As a vector field it is written:
$$\vec X = a(t)\vec X_1 + b(t)\vec X_2 + c(t)\vec X_3,$$
being,
$$\vec X_1 = \frac{\partial}{\partial x},\quad\vec X_2 =
x\frac{\partial}{\partial x},\quad\vec X_3 =
x^2\frac{\partial}{\partial x}.$$ The Lie algebra spanned by these
vector fields is the infinitesimal generator of the group
$PGL(1,\mathbb C)$ of projective transformations of the projective
line $\mathbb P(1,\mathbb C)$,
$$x\mapsto \frac{u_{11}x + u_{12}}{u_{21}x + u_{22}},$$
which is a Lie group of dimension 3. In order to make the
computation easier, let us consider the following: modulo a finite
group of order $2$, the group $PGL(1,\mathbb C)$ is identified with
$SL(2,\mathbb C)$ the group of $2\times 2$ matrices with determinant
$1$. By this isogeny the automorphic system is transformed into the
linear system,
\begin{equation}\label{EqRiccatiSL}
\frac{d}{dt}\left(\begin{matrix} u_{11} & u_{12} \\ u_{21} & u_{22}
\end{matrix}\right)= \left(\begin{matrix} \frac{b}{2} & a \\ -c & -\frac{b}{2} \end{matrix}\right)
\left(\begin{matrix} u_{11} & u_{12} \\ u_{21} & u_{22}
\end{matrix}\right)\end{equation}
each solution $(u_{ij}(t))$ induces the global solution of the
Riccati equation,
$$x(t) = \frac{u_{11}(t)x_0 + u_{12}(t)}{u_{21}(t)x_0 + u_{22}(t)}.$$
\end{example}

\subsection{Logarithmic Derivative} \index{logarithmic derivative!analytic}

For each open subset $S'$ of the Riemann surface $s$ we denote by
$\mathcal O(S',G)$ the space of analytic maps from $S'$ to $G$; the
elements of this space are complex analytic \emph{curves} in $G$.
For a curve $\sigma(t)\in \mathcal O(S,G)$ and a point $t_0$ in $S'$
we denote by $\sigma'(t_0)$ to its tangent vector at $t_0$, which is
the image of $\partial_{t_0}$ by the tangent morphism
$$\sigma'_{t_0}\colon T_{t_0}S\to T_{\sigma(t_0)}G.$$
As usually, we identify the Lie algebra $\mathcal R(G)$ with the
tangent space at the identity element $T_{Id}G$. There is only an
element of $\mathcal R(G)$ whose value at $\sigma(t_0)$ is
$\sigma'(t_0)$. The value of this \emph{right invariant} vector
field at $Id$ is $R'_{\sigma^{-1}(t_0)}(\sigma'(t_0))$. In such way
we can assign to $\sigma$ a map from  $S'$ to $\mathcal R(G)$ that
assigns to each $t_0\in S'$ the right invariant vector field whose
value at $\sigma(t_0)$ is $\sigma'(t_0)$. By the identification of
$\mathcal R(G)$ with $T_{Id}G$ this map sends $t$ to
$R_{\sigma(t)^{-1}}'(\sigma'(t))$. This is precisely Kolchin's
logarithmic derivative (see \cite{Ko1}).

\begin{definition}\label{C1DEF1.3.7}
  Let $\sigma\in \mathcal O(S',G)$ be a curve in $G$. We call logarithmic
derivative $l\partial(\sigma(t))$ of $\sigma(t)$ to the map from
$S'$ to $\mathcal R(G)$ that assigns to each $t_0\in S'$ the right
invariant vector field whose value at $\sigma(t_0)$ is
$\sigma'(t_0)$. The logarithmic derivative is a map,
\begin{eqnarray*}
l\partial\colon \mathcal O(S', G) &\to& \mathcal
R(G)\otimes_{\mathbb
C} \mathcal O(S'), \\
\sigma(t) &\to& l\partial(\sigma(t)) =
R_{\sigma(t)^{-1}}'(\sigma'(t)).
\end{eqnarray*}
\end{definition}

  Because of the construction of the logarithmic derivative the
  following result becomes self-evident.

\begin{proposition}
Let $\vec A$ be an automorphic system in $G$ depending on $S$. Then
$\sigma\in\mathcal O(S',G)$ is a solution of $\vec A$ if and only if,
\begin{equation}\label{ATeq24} l\partial(\sigma) = \vec
A - \partial \end{equation}
\end{proposition}

  Expression \eqref{ATeq24} is know as the \emph{automorphic
equation} \index{automorphic!equation}
 of the automorphic system $\vec A$. Solving the
automorphic vector field $\vec A$ is equivalent to finding a
particular solution for the automorphic equation.

 Let us recall that the adjoint automorphism $\Adj_{\sigma}$ is the tangent
map at the identity element of the internal automorphism of $G$,
$ \xi \mapsto \sigma \cdot \xi \cdot \sigma^{-1}.$
Logarithmic derivative satisfies the following property with respect to composition.

\begin{proposition}[Gauge change formula]\label{PropLDComposition}
 Consider $\sigma(t)$ and $\tau(t)$ in $\mathcal O(S,G)$. The
 composition $\sigma(t)\tau(t)$ is also an element of $\mathcal
 O(S,G)$. We have:
 $$l\partial(\sigma(t)\tau(t)) = l\partial(\sigma(t)) +
 Adj_{\sigma(t)}(l\partial(\tau(t))).$$
\end{proposition}

\begin{proof}
  By direct computation,
  $l\partial(\sigma(t)\tau(t)) =$
  $R'_{\tau(t)^{-1}\sigma(t)^{-1}}((\sigma(t)\tau(t))') =$\linebreak
  $R'_{\sigma(t)^{-1}}R'_{\tau(t)^{-1}}(R'_{\tau(t)}(\sigma'(t)) +
  L'_{\sigma(t)}(\tau'(t))) =$
  $R'_{\sigma(t)^{-1}}(\sigma'(t)) + $
  $R'_{\sigma(t)^{-1}}L'_{\sigma(t)}R'_{\tau(t)^{-1}}(\tau'(t)) = l\partial(\sigma(t)) +
 Adj_{\sigma(t)}(l\partial(\tau(t))).$
\end{proof}

\begin{corollary}
  For $\sigma(t)\in \mathcal O(S,G)$,
  $$l\partial(\sigma(t)^{-1}) = -
  Adj_{\sigma(t)^{-1}}(l\partial(\sigma(t))).$$
\end{corollary}

\begin{proof}
  Apply the gauge change formula to the composition $\sigma(t)\cdot
  \sigma(t)^{-1} = Id$.
\end{proof}

\begin{example}[Logarithmic derivative of matrices] We take the canonical basis
$\frac{\partial}{\partial u_{ij}}$ of the tangent space of $GL(
n,\mathbb C)$. Let us consider $U\in GL(n,\mathbb C)$; the tangent
space $T_U GL(n,\mathbb C)$ is identified with the space of $n\times
n$ complex square matrices. We identify the Lie algebra
$gl(n,\mathbb C)$ with the tangent space at $I$, the identity
matrix. Let $U(t)$ be a time-dependent non-degenerate matrix. Then
$U'(t) \in T_{U(t)}GL(n,\mathbb C)$ is the matrix whose coefficients
are the derivatives of coefficients of $U$. In order to identify it
with an element of $gl(n,\mathbb C)$ we have to apply a right
transformation,
  $$R'_{U^{-1}(t)} \colon T_{U(t)}GL(n,\mathbb C) \to T_{Id}GL(n,\mathbb C) =
gl(n,\mathbb C).$$
  As $R_{U^{-1}}$ is a linear map on the functions $u_{ij}$, then it
is its own differential; therefore $\frac{d\log (U(t))}{dt} =
R'_{U^{-1}(t)}U'(t) = U'(t)U^{-1}(t)$, and then
$$\frac{d\log(U)}{dt}= U'U^{-1}.$$

 For a time-dependent matrix $A$, that we consider as a curve in the Lie
algebra $gl(n,\mathbb C)$, we set the automorphic equation,
$$\frac{d\log U}{dt} = A.$$
It is equivalent to
$$ U'U^{-1} = A, \quad U' = AU,$$
the linear system defined by $A$.
\end{example}

\section{Lie's Reduction Method}

  Sophus Lie developed a method of reduction of Lie-Vessiot systems when
a particular solution is known. In despite of its lack of
popularity, this method is the hearth underlying of most known
methods of reduction of differential ordinary equations, as the
classical reduction of Riccati equation (see \cite{Darboux} vol I.
ch. I-IV), symplectic reduction (see \cite{Bryant1991} lec. 7),
representation formulas of solution of matrix differential equation
(see \cite{ShniderWinternitz}), generalized Wei-Norman method (see
\cite{Carinena2002}). Another interesting application is the
D'Alambert reduction of variational equations to normal variational
equations in dynamical systems (see \cite{MoralesBook}). It is also
equivalent to the method of reduction shown by Cari\~nena and Ramos
(cf. \cite{Carinena2001}). Here this method is presented by means of
gauge transformations and the automorphic equation.

\medskip

\emph{From now on let us consider the following objects: a
$G$-homogeneous space $M$, an automorphic system $\vec A$ in $G$
depending on the Riemann surface $S$, an origin point $x_0\in M$.
Denote by $H$ the isotropy subgroup of $x_0$, and by $\vec X$ the
induced Lie-Vessiot system $\vec A^M$ in $M$ by $\vec A$. }

\subsection{Gauge Transformations}

  We fiber the extended phase space $S\times G$ over the Riemann
  surface $S$. It is a trivial principal fiber bundle $\pi\colon
  S\times G\to S$. We perform the same operation for $M$; we
  consider then $\pi\colon S\times M \to M$ as an associated bundle of fiber $M$
(see \cite{Nomizu}).   
  A map $\sigma(t)\in\mathcal O(S,G)$ is considered
  as a section of $\pi$. This section induces an automorphism $L_{\sigma(t)}$ of
  the principal bundle,
  $$L_{\sigma(t)}\colon S \times G \to S \times G, \quad
  (t,\tau)\mapsto (t,\sigma(t)\cdot \tau),$$
  and an automorphism of the associated bundle that we denote by the
  same symbol,
  $$L_{\sigma(t)}\colon S \times M \to S \times M, \quad (t,x)
  \mapsto (t,\sigma(t)\cdot x).$$

\index{gauge transformation}
\begin{definition} The above automorphisms are called gauge
transformations induced by $\sigma(t)$.
\end{definition}

  This is nothing but Cartan's notion of \emph{rep\`ere mobile} on
  the bundle. These are the natural transformations for Lie-Vessiot
  systems. Gauge transformations are easily understand by terms of
  the logarithmic derivative.

\begin{theorem}
  $L_\sigma$ transforms automorphic systems onto automorphic
  vector systems, and Lie-Vessiot systems onto Lie-Vessiot system.
  A map $\tau(t)$ is a solution of the automorphic equation
  \eqref{ATeq24} if and only if $L_{\sigma(t)}(\tau(t)) = \tau(t)\cdot\sigma(t)$
  verifies,
  $$l\partial( \tau(t)\cdot \sigma(t) ) = Adj_{\sigma(t)}(\vec A - \partial) +
  l\partial(\sigma(t)).$$
\end{theorem}

\begin{proof}
  Assume that $\tau(t)$ is a solution of the equation \eqref{ATeq24}.
Then by the fundamental property of logarithmic derivative,
$l\partial(\sigma(t)\tau(t)) = Adj_{\sigma(t)}(\vec A - \partial) +
l\partial(\sigma(t))$. The \emph{``if and only if''} condition is
attained by considering the inverse gauge transform
$L_{\sigma(t)^{-1}}$. It proves that $L_{\sigma(t)}$ maps the
automorphic system $\vec A$ to the automorphic system $\vec B$
defined by,
$$\vec B_t = Adj_{\sigma(t)}(\vec A_t) +
  l\partial(\sigma(t)).$$
and then it maps also the Lie-Vessiot system $\vec A^M$ to $\vec
B^M$.
\end{proof}

\begin{example}
As it is well known, if we consider a linear system
$$x' = Ax$$
and a change of variable $z = Bx$, being $B$ a time dependent
invertible matrix, then $z' = B'x + Bx' = B'B^{-1}z + BAx = (B'
+BA)B^{-1}z$, and $z$ satisfies the transformed linear system,
$$z' = (B' + BA)B^{-1}z$$
where $(B' + BA)B^{-1} = B'B^{-1} + \Adj_B(A)$, as above.
\end{example}

\subsection{Lie's Reduction Method}

\emph{ Let us recall that we consider $x_0$ a point of the
$G$-homogeneous space $M$ as origin, and we denote by $H$ the
isotropy subgroup of $x_0$. We also denote by $H^0$ to the connected
component of the identity of $H$.}

  From the canonical inclusion of Lie algebras $\mathcal
R(H^0)\subset \mathcal R(G)$ we know that an automorphic system
$\vec B$ in $H^0$ is, in particular, an automorphic system in $G$.
The non-autonomous vector field $\vec B$ in $H^0$ naturally extends
to a non-autonomous vector field in $G$ by right translations. In
order to solve the extended non-autonomous vector field it is enough
to find a particular solution of $\vec B$ in $H^0$. Reciprocally, an
automorphic system in $G$ restricts to an automorphic system in
$H^0$ if and only if its Lie-Guldberg-Vessiot algebra is contained
in $\mathcal R(H^0)$. The Lie's method of reduction stands on the
following key lemma that characterizes which automorphic systems in
$G$ are, in fact, automorphic systems in $H^0$.

\begin{lemma}
  Assume that $x_0$ is a constant solution of $\vec A^M$. Then,
$\vec A$ is an automorphic system in $H^0$: there exist
right-invariant vector fields $B_i\in\mathcal R(H^0)$ such that:
$$\vec A = \partial + \sum_{i=1}^{s} f_i(t)\vec B_i.$$
\end{lemma}

\begin{proof}
  For each $t_0$ in $S$ we take a local solution $\sigma(t)$ of $\vec A$, defined
  in some neighborhood $S'$ of $t_0$, with initial condition $t_0\mapsto Id$. In $S'$
  we have $\vec A - \partial = l\partial(\sigma(t))$. As $\sigma(t)\cdot x_0 = x_0$,
  $\sigma(t)$ is a curve in $H$ and its
  logarithmic derivative takes values at $\mathcal R(H)$. The Lie
  algebra $\mathcal R(H)$ coincides with the Lie algebra $R(H^0)$
  of the connected component of the identity. Finally we conclude that for all
  $t_0\in S$, $\vec A_{t_0}\in
  \mathcal R(H)$. Then the Lie-Guldberg-Vessiot algebra of $\vec A$
  is included in $\mathcal R(H)$.
\end{proof}

  Let us examine the general case of reduction. Assume that we know an analytic
solution $x(t)$ for the Lie-Vessiot system $\vec A^M$. For each $x
\in M$ we denote,
$$H_{x_0,x} = \{\sigma\in G| \sigma\cdot x_0 = x\}.$$
The isotropy group $H$ acts in $H_{x_0,x}$ free and transitively by
the right side, therefore $H_{x_0,x}$ is a principal homogeneous
$H$-space.

\medskip

We construct the following sub-bundle of $\pi\colon S\times G \to
S$. We define  $\mathcal H\subset S\times G$, and $\pi_1\colon
\mathcal H \to S$ the restriction of $\pi$ in such way that for
$t_0\in S$ the stalk $\pi_1^{-1}(t_0)$ is $H_{x_0,x(t_0)}$. Then
$\pi_1$ is a principal bundle modeled over $H$. Let us take a
section $\sigma(t)$ of $\pi_1$ defined in some $S'\subset S$. Thus,
in $S'$ we have that $x(t) = \sigma(t)\cdot x_0$. Let us consider
the gauge transformation $L_{\sigma(t)^{-1}}$. It maps the
automorphic system $\vec A$ to an automorphic system $\vec B$,
$$\vec B-\partial = Adj_{\sigma(t)^{-1}}(\vec A-\partial - l\partial(\sigma)),$$
$L_{\sigma(t)^{-1}}(x(t))$ is a solution of $\vec B$. But,
$L_{\sigma(t)^{-1}}(x(t)) = \sigma^{-1}(t)\sigma(t)\cdot x_0 = x_0$.
Thus, we are in the hypothesis of the previous lemma. We have proven
the following result.

\begin{theorem}[Lie's reduction method]\label{C1THEliereduction}
Assume that there is a solution $x(t)$ of $\vec A^M$ defined in a
neighborhood of $t_0$. Then there exists a neighborhood $S'$ of
$t_0$ and a gauge transformation defined in $S'\times G$ that maps
the automorphic system $\vec A$ to an automorphic system $\vec B$ in
$H^0$.
\end{theorem}

\medskip
  For performing Lie's reduction we need a section of a principal
bundle. In general this bundle is not trivial, and then there are no
global sections. We have to consider two different cases: compact
and non-compact Riemann surfaces. For non-compact Riemann surfaces
we use the following result due to Grauert (see \cite{Sibuya}).

\begin{theorem}[Grauert theorem]
Let $S$ be a complex connected non-compact Riemann surface. Let
$F\to S$ be a locally trivial complex analytic principal bundle with
a connected complex Lie group as structure group. Then there is a
meromorphic section of $F$ defined in $S$.
\end{theorem}

  For compact Riemann surfaces we use the correspondence between
algebraic and analytic geometry. If $G$ is algebraic, then complex analytic principal
bundles modeled over $G$ are algebraic. In the general case there
are not meromorphic global section: algebraic bundles are not
locally trivial in Zariski topology. They are locally
\emph{isotrivial}; isomorphic to trivial bundles \emph{up to a
ramified covering}. It is convenient to introduce the following
class of groups.

\begin{definition}\label{DEFspecialgroup} 
  A complex analytic Lie group $H$ is called special if all principal
bundle modeled over $H$ has a global meromorphic section.
\end{definition}

There is close relation between Galois cohomology and special groups.
The first set of Galois cohomology $\mathcal H^1(H,K)$ classifies principal 
homogeneous $H$-spaces with coefficients in the field $K$ (see \cite{Ko1}). 
There is a dictionary
between meromorphic principal bundles over $S$ modeled over $H$ and principal 
homogeneous spaces with coefficients in the field $\mathcal M(S)$ of meromorphic
functions in $S$. Hence, an algebraic special group is an algebraic group whose
first Galois cohomology set with coefficients in any field of meromorphic
functions vanish. 

  Special groups are \emph{linear and connected} (see \cite{GAGA} th\'eor\`eme 1).
Fortunately, groups that appear in our integrability theory are
special groups. Let us cite the following result (\cite{MRS} theorem
8). 

\begin{theorem}
Let $Sp(2n,\mathbb C)$ be the symplectic group of $n$ degrees of
freedom.
\begin{enumerate}
\item[(i)] $Sp(2n,\mathbb C)$ is special.
\item[(ii)] Every connected solvable linear algebraic group is special.
\item[(iii)] Let $H\lhd G$ be a normal subgroup. If $H$ and $G/H$ are special then $G$ is special.
\end{enumerate}
\end{theorem}

  If the isotropy group $H$ of $x_0$ in $M$ is special or $S$ is open then the bundle
$\pi_1\colon\mathcal H \to S$ has a global meromorphic section
$\sigma$. The gauge transformation $L_{\sigma^{-1}}$ maps the
automorphic system $\vec A$ to an automorphic system in $H^0$. We
have proven the following result.

\begin{proposition}\label{C1PROmeromorphicliereduction}
Assume that there is a meromorphic solution $x(t)$ of $\vec A^M$
defined in $S$. Let us take a point $t_0\in S$ and denote by $x_0$
the point $x(t_0)$. Let $H$ be isotropy subgroup of $x_0$ in $M$.
Assume one of the following additional hypothesis,
\begin{enumerate}
\item[(a)] $H$ is a special group.
\item[(b)] $S$ is non compact and $H$ is connected.
\end{enumerate}
In such case there is a meromorphic gauge transformation in $S\times
G$ that reduces $\vec A$ to an automorphic system in $H^0$.
\end{proposition}

\begin{example}
Consider the Riccati equation
$$\frac{dx}{dt} = a(t) + b(t)x + c(t)x^2$$
and suppose that we know a particular solution $f$. Then let
$$\sigma(t) = \left(\begin{matrix} 1& f \\ 0 & 1 \end{matrix}\right)$$
so that $$f = \sigma(t)\cdot 0$$  using the linear fractional action
of $SL(2,\mathbb C)$ on the projective line $\mathbb P(1,\mathbb C)$
of the above example. The isotropy of $0$ is the subgroup $H_0$ of
matrices of the form:
$$H_0 = \left\{\left(\begin{matrix} \lambda & 0 \\ \mu & \lambda^{-1}
\end{matrix}\right)\right\}.$$
By means of the gauge transformation induced by $\sigma^{-1}$,
$$\left(\begin{matrix} v_{11} & v_{12} \\ v_{21} & v_{22} \end{matrix}\right) =
\left(\begin{matrix}u_{11} - fu_{21} & u_{12} - u_{22} \\ u_{21} &
u_{22} \end{matrix}\right).$$ we transform the linear system
\eqref{EqRiccatiSL} into the reduced system,
$$\frac{d}{dt}\left(\begin{matrix} v_{11} & v_{12} \\
v_{21} & v_{22}
\end{matrix}\right) = \left(\begin{matrix} \frac{b}{2}+cf & 0 \\ -c & -\frac{b}{2}-cf \end{matrix}\right)
\left(\begin{matrix} v_{11} & v_{12} \\ v_{21} & v_{22}
\end{matrix}\right)$$
which is in triangular form, and then integrable by quadratures. The
induced gauge transformation in the projective line maps $x$ to $z =
x - u$, and then $z$ verifies the Riccati equation,
$$\frac{dz}{dt} = (b + 2cf)z + cz^2$$
if we set $w = \frac{1}{z} = \frac{1}{x-u}$ we find the classical
transformation of the Riccati equation to inhomogeneous linear
equation (cf. \cite{Darboux} Vol I, chapter II).
$$\frac{dw}{dt} = -c - (2b + cf)w.$$
\end{example}

\section{Analytic Galois Theory}

Here we present a differential Galois theory for automorphic systems
in $G$ that depends meromorphically on a Riemann surface $S$. Our
approach is similar to the \emph{tannakian} presentation of
differential Galois theory. The difference is that here we use the
category of $G$-homogeneous spaces instead of constructions by
tensor products. In the tannakian approach, the Galois group
stabilizes all meromorphic invariant tensors of the differential
equation; equivalently in our approach the Galois group stabilizes
all meromorphic solutions of Lie-Vessiot systems induced by $\vec
A$. This presentation is even more direct than the classical
tannakian approach: we work directly in the category of
$G$-homogeneous spaces. The Galois group and its applications appear
naturally. In this frame there is no need of some technical points
that usually appear in other presentations of tannakian differential
Galois theory. We construct our Galois group as the fiber of a
geometrically defined object, the \emph{Galois bundle}.

\subsection{Galois Bundle}\label{SSGaloisBundle}

  Let us consider the automorphic system in $G$,
$$\vec A = \partial + \sum f_i(t)\vec A_i,$$
and assume that the $f_i(t)$ are meromorphic functions in $S$. Let
us consider $$S^\times = S \setminus \{\mbox{poles of }f_i\},$$ the
Riemann surface that we obtain from $S$ by removing the poles of the
meromorphic functions $f_i$. Then $\vec A$ is a complex analytic
automorphic system in $G$ depending on $S^\times$. For each
$G$-homogeneous space $M$ we consider $\vec A^M$, the induced
Lie-Vessiot system.

\medskip

For each $M$, let us define $\mathcal M_0(\vec A^M)$ as the set of
solutions of $\vec A^M$ defined in $S^\times$ that are
\emph{meromorphic} at $S$. Let $\mathfrak C(G)$ be the set of
conjugacy classes of closed analytic subgroups of $G$. To each
$\mathfrak c\in \mathfrak C(G)$ it corresponds an homogeneous space
$M(\mathfrak c)$ isomorphic to $G/H$ being $H$ any closed analytic
subgroup of $G$ whose class of conjugation is $\mathfrak c$. When
$\mathfrak c$ varies along $\mathfrak C(G)$, $M(\mathfrak c)$ varies
along the set of different class of isomorphic $G$-homogeneous
spaces. Finally, we define the set of meromorphic solutions
associated to $\vec A$,
$$\mathcal M(\vec A) = \bigcup_{\mathfrak c\in\mathfrak
C(G)}\mathcal M_0(\vec A^{M(\mathfrak c)})$$ The set $\mathcal
M(\vec A)$ consists of \emph{all} the different meromorphic solutions
of \emph{all} the different Lie-Vessiot systems induced in
$G$-homogeneous spaces.

\begin{definition}
For $t_0\in S^\times$ we define \emph{the analytic Galois group} of
$\vec A$ at $t_0$, $\Gal_{t_0}(\vec A)$ as the subgroup of $G$ that
stabilizes the values at $t_0$ of \emph{all} the meromorphic
solutions of \emph{all} the Lie-Vessiot systems induced by to $\vec
A$.
\begin{equation}\label{ATeq26}
\Gal_{t_0}(\vec A) = \bigcap_{x(t)\in \mathcal M(\vec A)}
H_{x(t_0)}.
\end{equation}
\end{definition}

The Galois group $\Gal_{t_0}(\vec A)$ is an intersection of closed
complex analytic subgroups of $G$, and then it is a closed analytic
subgroup of $G$.

\begin{lemma}\label{LM2.15}
  The class of conjugation of $\Gal_{t_0}(\vec A)$ does not depend
on $t_0\in S^\times$. Moreover $\Gal_{t_0}(\vec A)$ depends
analytically on $t_0$ in $S^\times$.
\end{lemma}

\begin{proof}
Consider $t_0$ and $t_1$ in $S^\times$. If $t_0$ and $t_1$ are close
enough we can assume that there is a solution $\sigma(t)$ of $\vec
A$ defined un a connected neighborhood including $t_0$ and $t_1$. By
a right translation we can assume that $\sigma(t_0)$ is the identity
element. Consider $\tau$ in $\Gal_{t_0}(\vec A)$. It means that for
all meromorphic solution $x(t)\in \mathcal M(\vec A)$, $\tau\cdot
x(t_0) = x(t_0)$. We have $x(t) = \sigma(t)\cdot x(t_0)$. Hence,
$x(t_1) = \sigma(t_1)\cdot x(t_0)$ and $\sigma(t_1)\cdot \tau\cdot
\sigma(t_1)^{-1}\cdot x(t_1) =
\sigma(t_1)\cdot\tau\cdot\sigma(t_1)^{-1}\sigma(t_1)\cdot x(t_0) =
\sigma(t_1)\cdot\tau\cdot x(t_0) = \sigma(t_1)\cdot x(t_0) =
x(t_0)$. Then, $\sigma(t_1)\cdot\tau\cdot\sigma^{-1}(t_1) \in
\Gal_{t_1}(\vec A)$ an we conclude that,
$$\Gal_{t_1}(\vec A) = \sigma(t_1)\cdot \Gal_{t_0}(\vec A)\cdot\sigma(t_1)^{-1},$$
the Galois groups at $t_0$ and $t_1$ are conjugated. It is proven
that the conjugacy class of the Galois group it is locally constant;
$S^\times$ is a connected Riemann surface therefore it is constant.

\medskip
Now, let us consider $H$ a subgroup of $G$ of the same class of
conjugation that $\Gal_{t_0}(\vec A)$. The normalizer subgroup
$$\mathcal N(H) = \{\sigma\in G |\,\,\sigma H = H\sigma \},$$
is the bigger intermediate group $H\subset Z \subset G$ such that
$H\lhd Z$. Let $M$ be the quotient $G$-homogeneous space $G/\mathcal
N(H)$, and let us consider the natural projection $\pi_1\colon G \to
M$. Points of $M$ parameterize the class of conjugation of $H$, to
$x = [\sigma]$ it corresponds the group $\sigma \cdot H \cdot
\sigma^{-1}$, and that  \emph{the natural action of $G$ on the
quotient is nothing but the action of $G$ by conjugation on the
conjugacy class of $H$}. This parametrization allow us to define a
map $h\colon S^\times \to M$ that sends $t\mapsto h(t)$. The image
$h(t_0)$ is the point of $N$ corresponding to the subgroup
$\Gal_{t_0}(\vec A)$ in its conjugacy class. Near $t_0$ take any
solution $\sigma(t)$ of $\vec A$ such that $\sigma(t_0)$ is the
identity. Then, $\Gal_t(\vec A) = \sigma(t)\cdot \Gal_{t_0}(\vec A)
\cdot \sigma(t)^{-1}$ or equivalently $h(t) = \sigma\cdot h(t_0)$.
Thus, $h(t)$ is a solution for the Lie-Vessiot system $\vec A^M$
induced in $M$, so that it is an analytic function in $S'$.

\end{proof}

\begin{lemma}\label{C1LEMgaloissolution}
  Consider $H$ the Galois group of $\vec A$ at $t_0$ and $M$ the quotient space $G/H$. Then $\vec A^M$
has a meromorphic solution in $M$. \end{lemma}

\begin{proof}
Consider $x_0=[H]$ the origin point of $M$. The group
$\Gal_{t_0}(\vec A)$ is the isotropy group of the values of all
meromorphic solutions of induced Lie-Vessiot systems, as stated in
formula \eqref{ATeq26}. The complex analytic group $G$ is of finite
dimension. The equations of $\Gal_{t_0}(\vec A)$ as subgroup of $G$
are:
$$\Gal_{t_0}(\vec A) = \{\sigma\in G\,\,|\,\,\sigma(x(t_0)) =
x(t_0)\ ,\,\forall x(t)\in \mathcal M(\vec A)\}.$$ Each
$x(t)\in\mathcal M(\vec A)$ gives us some of the equations of
$\Gal_{t_0}(\vec A)$. As a complex analytic manifold $G$ is of
finite dimension, thus $\Gal_{t_0}(\vec A)$ is defined by a finite
number of equations, at least locally. Then it suffices to consider
a finite number of such meromorphic solutions, %
$y_1(t)$,$\ldots$,$y_m(t)$ each one defined in a homogeneous space
$y_k\colon S^\times \to M_k$. For each $M_k$, we have
$\Gal_{t_0}(\vec A) \subset H_{y_k(t_0)}$. By fixing $y_k(t_0)$ as
the origin point, the homogeneous space $M_k$ is identified with the
quotient $G/H_{y_k(t_0)}$. We have a natural projection of
$G$-spaces $p_k\colon M\to M_k$ that maps the origin $x_0$ of $M$
onto $y_k(t_0)$. By considering the cartesian power of those
projections, we construct
$$\pi_1\colon M \to M_1\times \ldots \times M_m,$$
which is an injective morphism of $G$-spaces, that identifies $M$
with an orbit in the cartesian product. The image of the origin
point is $(y_1(t_0),\ldots, y_m(t_0))$. Finally, the meromorphic
solution of the Lie-Vessiot system in the cartesian power,
$(y_1(t),\ldots,y_n(t))$, is contained in the image of $\pi_1$, so
that it is a solution of $\vec A^M$ which is meromorphic in $S$.
\end{proof}

\begin{corollary}
  The Galois group $\Gal_{t}(\vec A)$ depends meromorphically on  $t\in S$.
\end{corollary}

\begin{proof}
  Let us recall the proof of the Lemma \ref{LM2.15}. Let us denote by $H$
the analytic differential Galois group in some point $t_0\in
S^\times$. The set of different subgroups of $G$ conjugated with
$H$, \emph{id est} the class of conjugation of $[H]$, is
parameterized by the homogeneous space $G/\mathcal N(H)$. Let $M$ be
the homogeneous space $G/H$. Let us consider the meromorphic
solution $x(t)$ in $M$ of the above lemma given by Lemma
\ref{C1LEMgaloissolution}. By construction of $x(t)$ the isotropy
$H_{x(t)}$ is the Galois group at $t$ in $S^\times$. $H$ is
contained in its normalizer $\mathcal N(H)$: there is a natural
projection of homogeneous $G$-spaces,
$$M \to G/\mathcal N(H_{x(t_0)}).$$
Let $y(t)$ be the projection of $x(t)$. For $t_0$ near $t$ the group
corresponding to $y(t)$ is precisely $\sigma(t)\cdot y(t_0)$ where
$\sigma(t)$ is the local solution of $\vec A$ with initial condition
for $t_0$ the identity element of $G$. Then, the isotropy group of
$y(t)$ is $\sigma(t)\cdot H_{x(t_0)}\cdot \sigma(t)^{-1}$ which is
the Galois group at $t$. This property prolongs to the whole Riemann
surface $S^\times$ and we see that $y(t)$ parameterizes the group
$\Gal_{t}(\vec A)$ into its class of conjugacy.
\end{proof}

  The Galois group $\Gal_t(\vec A)$ depends meromorphically of $S$. Thus,
  we can define an analytic sub-bundle $\Gal(\vec A) \subset S\times G$,
  which is meromorphic in $S$ and such that the fiber of $t_0\in S^\times$
  is precisely $\Gal_{t_0}(\vec A)$.

\begin{definition}\label{ATGaloisBundle}
  We call Galois bundle of $\vec A$ to the complex analytic in $S^\times$ and meromorphic
  in $S$ principal bundle,
  $$\Gal(\vec A) = \bigcup_{t\in S^\times} \Gal_t(\vec A) \xrightarrow{\quad\pi\quad} S^\times.$$
\end{definition}

\begin{remark}
  In the algebraic case, one can consider just algebraic subgroups of $G$ and then algebraic
homogeneous spaces. In such a case, we will obtain an algebraic Galois group. We fall in
the Kolchin's theory of G-primitive extensions, and the Galois group here coincides with
the Galois group of the attained strongly normal extension. This connection has been
partially examined in \cite{B2008}, chapter 4. 
\end{remark}

\begin{remark}
  Another advange of this geometric presentation of the theory, even in the algebraic
case, is that the Galois group is well defined as a unique differential algebraic
subgroup of $G\otimes_{\mathcal C}\mathcal M(S)$. This group is endowed with a natural
differential equation in addition to the automorphic system: the Lie-Vessiot equation
induced by the adjoint action of $G$ on itself. $G\otimes_{\mathcal C}\mathcal M(S)$
endowed with this Lie-Vessiot system is a \emph{differential algebraic group}. It is clear
that the Galois bundle can be seen as a differential algebraic subgroup of such
group. Then we have a canonical representation of the Galois group as a differential 
algebraic subgroup of $G\otimes_{\mathcal C}\mathcal M(S)$.
\end{remark}

\subsection{Analytic Galois Bundle and Picard-Vessiot Bundle}

For this section let us assume that $G$ is $GL(E)$, the group of
linear automorphisms of a complex finite dimensional vector space.
In this case, the considered automorphic system $\vec A$ is a system
of linear homogeneous differential equations with meromorphic
coefficients in $S$. Picard-Vessiot theory is developed for these
equations. We can define the \emph{algebraic Galois
group} of $\vec A$ following the tannakian formalism (see
\cite{RamisMartinet} and \cite{Vanderput}).

\medskip

The automorphic system $\vec A$ is seen as a meromorphic linear
connection $\nabla_{E\times S}$ in the vector bundle $E\times S \to
S$. Let us consider $\mathcal T$ the category spanned by $E$ trough
tensor products, arbitrary direct sums, and their linear subspaces.
Denote by $\mathcal T^\nabla$ the category of \emph{linear
connections} spanned by $(E\times S,\nabla_{E\times S})$ trough
tensor products, arbitrary direct sums, and linear subconnections.
The objects of $\mathcal T$ are linear subspaces of the tensor
spaces,
$$F \subset T_{n_1,\ldots,n_r}^{m_1,\ldots,m_r}(E) =
\bigoplus_{i=1}^r E^{\otimes n_i} \oplus (E^*) ^{\otimes m_i}.$$ The
objects of $\mathcal T^{\nabla}$ are linear subconnections of the
induced connections in the tensor bundles
\begin{equation}\label{subconnection}
(V,\nabla_V) \subset (T_{n_1,\ldots,n_s}^{m_1,\ldots,m_r}(E)\times
S, \nabla_{T_{n_1,\ldots,n_r}^{m_1,\ldots,m_r}(E)}).
\end{equation}

As before, we define the Riemann surface $S^\times$ by removing from
$S$ the poles of the coefficients of the differential equations.
Each point $t$ in $S^\times$ defines a \emph{fiber functor}:
  $$\omega_t \colon \mathcal T^\nabla \to \mathcal T$$
that sends vector bundles to their fibers in $t$. For a
subconnection $(V,\nabla_V)$ as in \eqref{subconnection}, the fiber
on $t\in S^\times$ is a linear subespace $V_t$ of
$T_{n_1,\ldots,n_s}^{m_1,\ldots,m_r}(E)$.

 It is known that the algebraic differential Galois group is the group of
automorphisms of the fiber functor $\omega_t$. The representation of
this differential Galois group depends on the base point $t$. In
this way we obtain a bundle, that we call the \emph{Picard-Vessiot}
bundle,
$$PV(\vec A)\to S^\times,$$
whose fiber in $t$ is the group $PV_t(\vec A)$ of automorphisms of
$\omega_t$.

\medskip

  We represent the algebraic Galois group $PV_t(\vec A)$ into $GL(E)$
in the following way. Let us consider $(V,\nabla_V)$ an object of
the category $\mathcal T^\nabla$. Its fiber in $t\in S$ is a vector
subspace of certain tensor product
$$T_{n_1,\ldots,n_r}^{m_1,\ldots,m_r}(E) = \bigoplus_{i=1}^r
E^{\otimes n_i} \otimes (E^*) ^{\otimes m_i}.$$ An element
$\sigma\in GL(E)$ induces linear transformations of the tensor
product. It is known that $\sigma$ is in the differential Galois
group $PV_t(A)$ if and only if is stabilizes the vector space $V_t$
for all object $(V,\nabla)$ of the category $\mathcal T^\nabla$.
$$PV_t(\vec A) = \{\sigma\in GL(E) \,\,|\,\, \sigma(V_t) = V_t
\,\,\,\forall (V,\nabla_V)\in \mathcal Obj(\mathcal T^\nabla)\}.$$

There is a dictionary between linear connections $(V,\nabla_V)$ and
meromorphic solutions of Lie-Vessiot systems in associated to $\vec
A$ in algebraic homogeneous spaces.

Let us consider $(V,\nabla_V)$ as above. Let $k$ be the dimension of
the fibers of $V$; for each $t\in S^\times$ the fiber $V_t$ is a
$k$-plane of $T_{n_1,\ldots,n_r}^{m_1,\ldots,m_r}(E)$. Let us
consider the grassmanian variety $\Gr(k,
T_{n_1,\ldots,n_r}^{m_1,\ldots,m_r}(E))$ of $k$-planes in
$T_{n_1,\ldots,n_r}^{m_1,\ldots,m_r}(E)$; it is a $GL(E)$-space. The
map,
$$S^\times \to \Gr(k,T_{n_1,\ldots,n_r}^{m_1,\ldots,m_r}(E)) \quad t \mapsto V_t,$$
is a meromorphic solution of the Lie-Vessiot system induced by $\vec
A$ into the grassmanian variety. This solution is contained in a
$GL(E)$-orbit, that we denote by $M$. Hence, the map
$$S^\times \to M, \quad t\mapsto V_t,$$
is a meromorphic solution of the Lie-Vessiot system $\vec A^M$.

\medskip

  This homogeneous space $M$ is isomorphic to the quotient
$GL(E)/H_{V_t}$ where $H_{V_t}$ is the  \emph{stabilizer} subgroup
of the linear subspace $V_t$; it is an algebraic subgroup of
$GL(E)$. Reciprocally, by Chevalley's theorem, any algebraic 
group is the stabilizer of certain
vector subspace. This means that the \emph{algebraic
Galois group} $PV_t(\vec A)$ is the group of linear transformations
$\sigma\in GL(E)$ that fix the values in $t$ of all meromorphic
solutions of associated Lie-Vessiot systems in \emph{algebraic
homogenous spaces}. We have proven the following.

\begin{theorem}\label{TH2.18}
  There is a canonical inclusion $\Gal_t(\vec A)\subset PV_t(\vec A)$. The analytic
Galois group is Zariski dense in the algebraic Galois group.
\end{theorem}

\begin{proof}
  Let $\mathcal A(\vec A)$ be the set of \emph{all} the different meromorphic
solution of \emph{all} the different Lie-Vessiot systems induced by
$\vec A$ in \emph{algebraic} homogeneous $GL(E)$-spaces. Then,
$\mathcal A(\vec A) \subseteq \mathcal M(\vec A)$. We have that,
$$\Gal_{t}(\vec A) = \bigcap_{x\in \mathcal M(\vec A)} H_{x(t)}
\subseteq PV_t(\vec A) = \bigcap_{x\in \mathcal A(\vec A)}
H_{x(t)}.$$

Let us see that $\Gal_t(\vec A)$ is Zariski dense. Let $H$ be the
Zariski closure of $\Gal_t(\vec A)$. It is an intermediate algebraic
subgroup,
  $$\Gal_t(\vec A) \subseteq H \subseteq PV_t(\vec A).$$
Let $M$ be $G/H$. By Lemma \ref{C1LEMgaloissolution} there is a
meromorphic solution of the Lie-Vessiot system in $GL(E)/Gal_t(\vec
A)$; the algebraic homogeneous space $M$ is a quotient of such
space. Therefore, there is a meromorphic solution of $\vec A^M$ and
$PV_t(\vec A) \subseteq H$.
\end{proof}

\begin{example}
  Consider the differential equation,
  $$\dot x = \frac{1}{t}.$$
  It is an automorphic equation in the additive group $\mathbb C$.
  There is an analytic action of $\mathbb C$ on $\mathbb C^*$,
  $\mathbb C\times \mathbb C^*\to \mathbb C^* \quad (x,y)\mapsto
  x\cdot y = e^x y.$
  The associated Lie-Vessiot system is $\dot y = y/t$. It has
  meromorphic solutions $y = \lambda t$. The analytic Galois is contained
  in the isotropy group $2\pi i \mathbb Z \subset \mathbb
  C$; in fact they coincide. However, the algebraic Galois group is the whole additive group.
\end{example}

\subsection{Integration by Quadratures}

The Lie's reduction method, applied to an specific case of
homogeneous space, gives us an analytic version of Kolchin theorem
on reduction to the Galois group.

\begin{theorem}\label{ATkolchin}
 Assume that the fiber of the Galois bundle  $\pi\colon \Gal(\vec A)\to S^\times$,
is contained a connected group $H\subset G$. Assume one of the
additional hypothesis:
\begin{enumerate}
\item[(1)] $H$ is an special group.
\item[(2)] $S$ is a non-compact Riemann surface.
\end{enumerate}
Then there is a meromorphic gauge transform of $G\times S$ that
reduces $\vec A$ to an automorphic system in $H$.
\end{theorem}

\begin{proof} Let $M$ be the homogeneous space $\Gal_{t_0}(\vec A)$.
We can apply an internal automorphism of $G$ in order to ensure that
$\Gal_{t_0}(\vec A) \subset H$. There is a natural projection $M \to
G/H$. Because of Lemma \ref{C1LEMgaloissolution}, there is a
meromorphic solution of  $\vec A^M$. This solution projects onto a
meromorphic solution in $G/H$. By Lie's reduction method, Theorem
\ref{C1THEliereduction}, there exist a gauge transformation reducing
$\vec A$ to $\mathcal R(H)$.
\end{proof}

\subsection{Quadratures in Abelian Groups}

  If $G$ is a connected abelian group, it is known that the exponential map,
$$\mathcal R(G) \to G, \quad \vec A \mapsto \exp(\vec A),$$
is the universal covering of $G$; in fact it is a group morphism if
we consider the Lie algebra $\mathcal R(G)$ as a vector group.
 The integration of an automorphic equation in the vector space
$\mathcal R(G)$ is done by a simple quadrature in $S$; thus the
integration of an automorphic equation in $G$ is done by composition
of the exponential map with this quadrature:
$$\sigma(t) = \exp\left(\int_{t_0}^t \vec A(\tau)d\tau\right),$$
where $d\tau$ is the meromorphic $1$-form in $S$ such that $\langle
d\tau, \partial\rangle = 1$.

\subsection{Solvable Groups}

Assume that there is a subgroup $G'\lhd G$ such that the quotient
$\bar G = G/G'$ is an abelian group. We have an exact sequence of
groups,
$$G' \to G \to \bar G.$$
The automorphic vector field $\vec A$ on $G$ is projected onto an
automorphic vector field $\vec B$ on $\bar G$. $\bar G$ is a
abelian, and then we can find the general solution of $\vec B$ by
means of the exponential of a quadrature. The quadrature is of the
form
$$\int_{t_0}^t \vec B(\tau)d\tau,$$
where $\vec B(\tau)d\tau$ is a closed $1$-form with vectorial values
in $\mathcal R(G)$. This $1$-form is holomorphic in $S^\times$, and
meromorphic in $S$. In the general case this closed $1$-form is not
exact. We need to consider the universal covering $\tilde S^\times
\to S^\times$. $\tilde S^\times$ is simply connected, and by
\emph{Poincare's lemma} every closed $1$-form is exact. Then we can
define,
$$\sigma\colon \tilde S^{\times}\to \bar G, \quad t
\mapsto \exp\left(\int_{t_0}^t \vec B(\tau)d\tau\right).$$

Let $t_0$ be a point of $S^\times$ and $\tilde t_0$ a point of
$\tilde S^\times$ in the fiber of $t_0$. There is natural action of
the \emph{Poincare's fundamental group} $\pi_1(\tilde
S^{\times},\tilde t_0)$ on the space of sections $\mathcal O(\tilde
S^{\times}, G)$; this is the \emph{monodromy representation}. Let
$H_{\sigma}$ be the isotropy of $\sigma$ for this action. There is a
minimal intermediate covering $S^\times(\sigma)$ such that the
section $\sigma$ factorizes. The fiber of the such covering
$S^\times(\sigma)\to S^\times$ is isomorphic to the quotient
$\pi_1(\tilde S^\times,\tilde t_0)/H_\sigma$.
$$\xymatrix{& G \\ \tilde S^\times \ar[rr]\ar[rd]\ar[ru]^-{\sigma} & & S^\times(\sigma)  \ar[lu]_-{\sigma} \ar[ld] \\
& S^\times }$$

 Some of the ramification points $S^\times(\sigma)\to S^\times$ have finite
index; we add them to $S(\sigma)$ obtaining a bigger surface
$S_1(\sigma)$. The projection of $S_1(\sigma)$ onto $S$ is a
ramified covering of certain intermediate surface $S^\times \subset
S_1 \subset S$. The section $\sigma(t)$ is meromorphic in
$S_1(\sigma)$. Then, we substitute the Riemann surface $S_1(\sigma)$
for $S$; $\sigma(t)$ is a meromorphic solution of $\vec B$ in
$S_1(\sigma)$. We apply Lie's reduction method
\ref{C1PROmeromorphicliereduction}, and reduce our equation to an
automorphic equation in $G'$ with meromorphic coefficients in
$S_1(\sigma)$. We can iterate this process and we obtain then the
following theorem:

\begin{theorem}\label{ATsolvable}
Assume that $G$ is a connected solvable group, and one of the
following hypothesis:
\begin{enumerate}
\item[(a)] $G$ is a special group.
\item[(b)] $S$ is a non-compact Riemann surface.
\end{enumerate}
Then the automorphic system $\vec A$ on $G$ is integrable by
quadratures of closed meromorphic $1$-forms in $S$ and the
exponential map in $G$.
\end{theorem}

\begin{proof}
Consider a resolution chain $G_0 \lhd G_1 \lhd \ldots \lhd G_{n-1}
\lhd G$. We can consider the process above with respect $G_{n-1}
\lhd G$. If $S$ is non-compact, we are under the hypothesis of
Grauert theorem. In the compact case, if $G$ is special, then it is
a connected linear solvable group, so $G_{n-1}$ is also special. In
both cases we can apply Proposition
\ref{C1PROmeromorphicliereduction}. We reduce the automorphic system
to an automorphic system in $G_{n-1}$ and take coefficients
functions in the corresponding covering of $S$. We iterate this
process until we reduce $\vec A$ to canonical form $\partial$.
\end{proof}

\section{Infinitesimal Symmetries}

  Let us consider the extended phase space $S^\times\times G$ for the automorphic
system $\vec A$. We are looking for vector field symmetries of the
system. It means, vector fields $\vec L$ in $S^\times\times G$ such
that the Lie bracket verifies $[\vec L,\vec A] =
\lambda(t,\sigma)\vec A$. This equation defines a sheaf of Lie
algebras of infinite dimension of vector fields in $S^\times \times
G$. As stated in \cite{Ath1997}, we can differentiate between
characteristic and non-characteristic symmetries. Characteristic
symmetries are them that are tangent to the solutions, \emph{id est}
proportional to $\vec A$. The sheaf of Lie algebra of characteristic
symmetries is generated by $\vec A$: for any complex analytic
function $F$ in $S^\times\times G$, $F\vec A$ is a characteristic
symmetry of $\vec A$. Characteristic symmetries form a sheaf of
ideals of the sheaf symmetries, and then there is a quotient, the
\emph{sheaf of non-characteristic symmetries} (see \cite{Ath1998},
also \cite{Ath1997} for the linear case). On the other hand, we can
also consider the sheaf of Lie algebras of \emph{transversal
symmetries}. We say that a vector field in $S^\times\times G$ is
transversal if it is tangent to the fibers of the projection onto
$S$. A vector field is transversal if and only if it can be written
in the form,
\begin{equation}\label{ATeq27}
\vec L = \sum_{i=1}^s F_i(t,\sigma)\vec A_i,\quad F_i\in \mathcal
O_{S^\times\times G},
\end{equation}
where the $\vec A_i$ form a basis of the tangent bundle to $G$. For
example, they can form a basis of the Lie algebra $\mathcal R(G)$ of
right invariant vector fields; or alternatively, they can form a
basis of the Lie algebra $\mathcal L(G)$ of left invariant vector
fields in $G$. Both cases lead as to interesting conclusions.

Let $\vec L$ be a transversal symmetry of $\vec A$; direct
computation gives that the Lie bracket $[\vec L,\vec A]$ is a
transversal vector field. Then, transversal symmetries are defined
by the more restrictive equations,
$$\langle \vec L, dt\rangle = 0, \quad [\vec L,\vec A]= 0. $$

Any symmetry can be reduced to a transversal symmetry by adding a
multiple of $\vec A$. For any non-transversal symmetry $\vec L$, we
have that $\vec L - \langle dt, \vec L\rangle \vec A$ is a
transversal symmetry. The kernel of such projection is precisely the
space of characteristic symmetries. Thus, \emph{the sheaf of Lie
Algebras of non-characteristic symmetries is isomorphic with the
sheaf of transversal symmetries.} For this reason, we restrict our
studies to the sheaf of transversal symmetries.

\index{transversal symmetries}
\begin{definition}
We denote by $Trans(\vec A)$ the sheaf of Lie algebras of
transversal symmetries of $\vec A$.
\end{definition}

\subsection{Equation of Transversal Symmetries in Function of Left Invariant Vector Fields}

Let $\vec L$ be an analytic vector field in $G$; $\vec L$ is a left
invariant vector field if and only if $[\vec L, \vec R] = 0$ for all
right invariant vector field $\vec R \in \mathcal R(G)$.

  The right invariant vector fields are symmetries of the left
invariant vector fields and viceversa. This property lead us to some
interesting conclusions.

  Let us consider $\{\vec L_1,\ldots,\vec L_s\}$ a basis of $\mathcal L(G)$, the Lie algebra of
left invariant vector fields in $G$. An arbitrary transversal vector
field in $S\times G$ is written in the form $$\vec L = \sum_{i=1}^s
g_i(\sigma,t)\vec L_i.$$ Let us set out the equations of transversal
symmetries,
$$[\vec L, \vec A]=0,$$
we expand the Lie bracket,
$$[\vec L, \vec A] = - \sum_{j=1}^s(\vec
A g_j(t,\sigma))\vec L_j,$$ and then we obtain the equation for the
coefficients $g_j(t,\sigma)$,
$$\vec Ag_j(\sigma,t)=0,\quad j=1,\ldots,s.$$
We have proven the following result, which is an extension of a
result for the linear case due to Athorne \cite{Ath1997}.
\begin{theorem}\label{C1THE1.5.9}
  Consider $\mathcal O_{S\times G}^{\vec A}$ the sheaf of first integrals of $\vec A$.
Then, transversal symmetries of $\vec A$ are left invariant vector
fields with coefficients in $\mathcal O_{S\times G}^{\vec A}$,
$$Trans(\vec A) = \mathcal L(G)\otimes_{\mathbb C}\mathcal O^{\vec A}_{S\times G}.$$
\end{theorem}

Note that the algebra of left invariant vector fields is a finite
dimensional Lie algebra of dimension $s$ contained in $Trans(\vec
A)$.

\subsection{Equation of Transversal Symmetries in Function of Right Invariant Vector Fields}

Consider,
$$\vec A = \partial + \sum_{i=1}^s f_i(t)\vec A_i, \quad \vec L = \sum_{i=1}^s g_i(t,\sigma)\vec A_i,$$
where $\{\vec A_1,\ldots, \vec A_s\}$ is a basis of $\mathcal R(G)$.
Let us denote by $c_{ij}^k$ the constants of structure of the Lie
algebra $\mathcal R(G)$, $[\vec A_i,\vec A_j] = \sum_{k=1}^s
c_{ij}^k\vec A_k$.

Let us write the equations for transversal symmetries,
$$[\vec L,\vec A] = \sum_{k=1}^s\left( - \partial g_k(t,\sigma) -
\sum_{i=1}^k f_i(t)\left(\vec A_ig_k(t,\sigma) - \sum_{j=1}^s
g_j(t,\sigma)c_{ij}^k\right)\right)\vec A_k = 0,$$ that gives us the
$s$ partial differential equations satisfied by the coefficients
$g_k(\sigma,t)$ of $\vec L$,
\begin{equation}\label{ATeq28}
\partial g_k =  -
\sum_{i=1}^s f_i(t)\left(\vec A_ig_k - \sum_{j=1}^s
g_jc_{ij}^k\right).
\end{equation}

\subsection{Right Invariant Symmetries}

Consider $\mathcal R(G)$ as a $\mathbb C$-vector space. The group
$G$ acts in $\mathcal R(G)$ through the adjoint representation,
$$G \times \mathcal R(G) \to \mathcal R(G),\quad (\sigma,A)\mapsto Adj_{\sigma}(A) = L_{\sigma}A,$$
where $Adj_{\sigma}(A)$ is the vector field $A$ altered by a left
translation of ratio $\sigma$. Note that the value of
$Adj_{\sigma}(A)$ at the identity is
$R_{\sigma^{-1}}'L_\sigma'(A_{Id})$. By the adjoint representation,
$\mathcal R(G)$ is a $G$-space. Then the automorphic system $\vec A$
induces a Lie-Vessiot system $\vec R$ in $\mathcal R(G)$. Let us
analyze what is the nature of the solutions of $\vec R$.  A local
solution of $\vec R$ defined in $S'\subset S$ is an analytic map
$S'\to \mathcal R(G)$. We interpret this map as a vector field in
$S'\times G$ tangent to the fibers of the projection onto $S'$. Such
a solution is written in form,
$$\vec V(t) = \sum_{i=1}^n g_i(t)A_i,$$
and the differential equations for the coefficients $g_i(t)$ are,
\begin{equation}\label{ATeq29}
\partial g_k(t) =
\sum_{i,j=1}^s f_i(t)g_j c_{ij}^k
\end{equation}
which is precisely a particular case of equation \eqref{ATeq28}.
Then we can state:

\begin{lemma}
The solutions of $\vec R$ in $\mathcal R(G)\times S$ are the
transversal symmetries of $\vec A$ whose restriction to fibers of
$G\times S \to S$ are right-invariant vector fields.
\end{lemma}

Let us consider, as before, that $\vec A$ is meromorphic in $S$
complex analytic in $S^\times$. Denote $Right(\vec A)$ the set of
transversal symmetries of $\vec A$ meromorphic in $G\times S$ whose
restriction to fibers of $G\times S \to S$ are right invariant
vector fields. In other words, the space of \emph{meromorphic
solutions} of $\vec R$. The space $Right(\vec A)$ is a finite dimensional Lie
subalgebra of the algebra of sections of $Trans(\vec A)$ of
dimension less or equal than $s$. On the other hand, for each $t\in
S$ we can consider $Right_{t}(\vec A)$, the space of the values at
$t$ of elements of $Right(\vec A)$. We know that $Right_{t}(\vec A)$
is a Lie subalgebra of $\mathcal R(G)$, and its class of conjugacy
depends meromorphically on $t$ in $S$.

\begin{theorem}\label{C1THE1.5.11}
For all $t$ in $S^\times$ the group $\Gal_t(\vec A)$ is contained in
the centralizer of $Right_t(\vec A)$.
\end{theorem}

\begin{proof}
Consider $\vec X_1(t),\ldots \vec X_r(t)$ a basis of $Right(\vec
A)$. Then $\vec X_i(t)$ is a set of meromorphic solutions of the
adjoint equation induced by $\vec A$ in $\mathcal R(G)$. For each
$t\in S^\times$ we have that $\sigma\in \Gal_{t}(\vec A)$ verifies,
$$Adj_{\sigma}(\vec X(t)) = \vec X(t),$$
and then $\sigma$ is in the centralizer of $Right_{t}(\vec A)$.
\end{proof}

\begin{corollary}
\label{C1COR1.5.12}
  Let us $\vec A$ automorphic system in the symplectic group $Sp(2n,\mathbb
  C)$. If $Right(\vec A)$ contains an abelian algebra of dimension
  $n$, then for all $t$ in $S^\times$ the component of the identity element of the
  analytic Galois group $\Gal^0_t(\vec A)$ is abelian.
\end{corollary}

\begin{proof}
Let us consider the Lie algebra $sp(2n,\mathbb C)$. It is the Lie
algebra of linear Hamiltonian autonomous vector fields in $\mathbb
C^{2n}$. Consider $P$ the space of homogeneous polynomials of degree
$2$ in the canonical coordinates in $\mathbb C^{2n}$. The space $P$
is a Poisson algebra and \emph{Hamilton equations} gives us an
isomorphism of $P$ with $sp(2n,\mathbb C)$. For each $t$ in
$S^\times$ we can consider both the Lie algebra $gal_t(\vec A)$ of
the Galois group and $Right_t(\vec A)$ as Poisson subalgebras of
$P$. Theorem \ref{C1THE1.5.11} implies that $\Gal_{t}(\vec A)$ is
contained in the centralizer of $Right_{t}(\vec A)$, and it implies
that the Lie algebra of the Galois group $gal_{t}(\vec A)$ commutes
with $Right_t(\vec A)$. In terms of Poisson brackets:
$$\{gal_{t}(\vec A), Right_{t}(\vec A)\}=0.$$
Thus, $gal_{t}(\vec A)$ and $Right_t(\vec A)$ are orthogonal Poisson
subalgebras of $P$. If we assume that $Right_{t}(\vec A)$ is an
abelian subalgebra of $P$ of maximal dimension $n$ then, by
\cite{MoralesRamis1}, the orthogonal of an abelian subalgebra of
maximal dimension is also abelian. Hence $gal_{t_0}(\vec A)$ is
abelian, and the connected component $\Gal_{t_0}(\vec A)$ is an
abelian group.
\end{proof}

\begin{remark}
  The adjoint action of $G$ in $\mathcal R(G)$ is algebraic. In the case
of the Picard-Vessiot theory, Theorem \ref{C1THE1.5.11} also holds.
The group $PV_t(\vec A)$ is contained in the centralizer of
$Right_t(\vec A)$. It implies an stronger version of Corollary
\ref{C1COR1.5.12}, because the connected component of the algebraic
differential Galois group $PV^0_t(\vec A)$ is bigger that
$\Gal^0_t(\vec A)$.
\end{remark}

\section*{Acknowledgements}

  This research of both authors has been partially financed
by MCyT-FEDER Grant MTM2006-00478 of spanish goverment. The
first author is also supported by {\sc Civilizar}, the research 
agency of Universidad Sergio Arboleda. We also acknowledge 
prof. J.-P. Ramis and prof. E. Paul for their support 
during the visit of the first author to Laboratoire Emile Picard. 
We are also in debt with J. Mu\~noz of Universidad de Salamanca for 
his help with the original work of S. Lie. We thank also P. Acosta,
T. Lazaro and C. Pantazi who shared with us
the seminar of algebraic methods in differential equations in Barcelona.

\bibliographystyle{amsalpha}

\bigskip

{\sc\noindent David Bl\'azquez-Sanz \\
Escuela de Matem\'aticas\\
Universidad Sergio Arboleda\\
Calle 74, no. 14-14 \\
Bogot\'a, Colombia\\
}
E-mail: {\tt david.blazquez-sanz@usa.edu.co}

\bigskip

{\sc\noindent Juan Jos\'e Morales-Ruiz \\
Departamento de Inform\'atica y Matem\'aticas\\
Escuela de Caminos Canales y Puertos\\
Universidad Polit\'ecnica de Madrid
Madrid, Espa\~na\\
}
E-mail: {\tt juan.morales-ruiz@upm.es}

\end{document}